\documentclass[12pt,reqno,draft]{article} 
\usepackage{amsmath,amssymb,amsthm,amsfonts, indentfirst}
\usepackage{enumerate,color,bm}
\usepackage{mathtools}
\makeatletter
\def\@cite#1#2{[{{\bfseries #1}\if@tempswa , #2\fi}]}
\renewcommand{\section}{%
\@startsection{section}{1}{\z@}
{0.5truecm plus -1ex minus -.2ex}%
{1.0ex plus .2ex}{\bfseries\large}}
\def\@seccntformat#1{\csname the#1\endcsname.\ }
\makeatother

\numberwithin{equation}{section} 
\theoremstyle{theorem}
\newtheorem{thm}{Theorem}[section]

\newtheorem{lem}[thm]{Lemma}

\theoremstyle{definition}

\newtheorem*{prth1.1}{Proof of Theorem 1.1}
\newtheorem*{prth1.1c}{Proof of Theorem 1.1 (continued)}
\newtheorem*{prcor1.2}{Proof of Corollary 1.2}
\newtheorem*{prth1.3}{Proof of Theorem 1.3}

\newcommand{\ep}{\varepsilon}
\newcommand{\pa}{\partial}
\newcommand{\Rn}{\mathbb{R}^n}
\newcommand{\R}{\mathbb{R}}

\newcommand{\lam}{\lambda}
\newcommand{\cl}[1]{{\overline#1}}
\newcommand{\Tmax}{T_{\rm max}}

\newcommand{\I}{\mathcal{I}}
\newcommand{\tmax}{T_{\rm max}}

\begin{document}
\footnote[0]{
    2020{\it Mathematics Subject Classification}\/. 
    Primary: 35A01; 
    Secondary: 35Q92, 92C17.
    %35A01: Existence problems: global existence, local existence, non-existence
    %35Q92: PDEs in connection with biology and other natural sciences
    %92C17: Cell movement (chemotaxis, etc.)
    }
%==========================title==========================
\begin{center} 
    \Large{{\bf 
    Global existence and boundedness 
    in\\
    a chemotaxis-convection model with\\ sensitivity functions for tumor angiogenesis
    }
    }
\end{center}
\vspace{5pt}
%===========================author=========================
\begin{center}
    Yutaro Chiyo\footnote{Corresponding author.}\footnote{Partially supported by JSPS KAKENHI Grant Number JP22J11422.}\\
    %\\[2mm]
    \vspace{2mm}
    Department of Mathematics, 
    Tokyo University of Science\\
    1-3, Kagurazaka, Shinjuku-ku, 
    Tokyo 162-8601, Japan\\
    \vspace{6mm}

    Masaaki Mizukami\\%\footnote{Corresponding author.}\\
    \vspace{2mm}
    Department of Mathematics, Faculty of Education, 
    Kyoto University of Education \\
    1, Fujinomori, Fukakusa, Fushimi-ku, Kyoto 
    612-8522, Japan\\
%    \vspace{6mm}
%
%    %\\[2mm]
%    Tomomi Yokota%
%      \footnote{Partially supported by Grant-in-Aid for
%      Scientific Research (C), No.\,21K03278.}
   \footnote[0]{
    E-mail: 
    {\tt ycnewssz@gmail.com}, %}\footnote[0]{
    {\tt masaaki.mizukami.math@gmail.com}
    }\\
%    \vspace{2mm}
%    Department of Mathematics, 
%    Tokyo University of Science\\
%    1-3, Kagurazaka, Shinjuku-ku, 
%    Tokyo 162-8601, Japan\\
%%    \vspace{12pt}
    \vspace{2pt}
\end{center}
\begin{center}    
    \small \today
\end{center}

\vspace{2pt}
%=====================  Abstract  =======================
\newenvironment{summary}
{\vspace{.5\baselineskip}\begin{list}{}{%
     \setlength{\baselineskip}{0.85\baselineskip}
     \setlength{\topsep}{0pt}
     \setlength{\leftmargin}{12mm}
     \setlength{\rightmargin}{12mm}
     \setlength{\listparindent}{0mm}
     \setlength{\itemindent}{\listparindent}
     \setlength{\parsep}{0pt}
     \item\relax}}{\end{list}\vspace{.5\baselineskip}}
\begin{summary}
{\footnotesize {\bf Abstract.}
    This paper deals with the fully parabolic chemotaxis-convection model with sensitivity functions for tumor angiogenesis, 
%
%%============  problem  ============
    \begin{align*}
        \begin{cases}
        %==========  equations  ==========
        u_t=\Delta u-\nabla \cdot (u\chi_1(v)\nabla v)
                         +\nabla \cdot (u\chi_2(w)\nabla w),
        &x \in \Omega,\ t>0,
    \\[1.05mm]
        v_t=\Delta v+\nabla \cdot (v\xi(w)\nabla w)+\alpha u-\beta v,
        &x \in \Omega,\ t>0,
    \\[1.05mm]
        w_t=\Delta w+\gamma u-\delta w,
        &x \in \Omega,\ t>0
        \end{cases}
    \end{align*}
    under homogeneous Neumann boundary conditions and initial conditions, 
where $\Omega \subset \mathbb{R}^n$ $(n \le 3)$ is a bounded domain with
smooth boundary, 
$\chi_1, \chi_2, \xi$ are functions satisfying some conditions and 
$\alpha, \beta, \gamma, \delta>0$ are constants. 
%Also,  satisfying some conditions; 
%a typical example is the function $\frac{a}{(b+s)^k}$, where $a>0, b \ge 0, k>1$. 
%When $\chi_1, \chi_2, \xi$ are constants, 
%global existence and boundedness of classical solutions 
%to the system have already been obtained by Li--Tao 
%(J. Math.\ Anal.\ Appl.; 2020;481;123474) and Tao--Winkler 
%(Nonlinear Anal.; 2021;208;112324). 
%However, in the case that $\chi_1, \chi_2, \xi$ are functions 
%the system has not been analysed yet. 
The purpose of this paper is to establish global existence and boundedness in this system. 
%by determining conditions for $\chi_1, \chi_2, \xi$. 
%when $\chi_1, \chi_2, \xi$ are functions. 
%It is shown that there exists a unique global solution which is uniformly bounded
%under some smallness conditions for $\chi_1, \chi_2, \xi$. 
}
\end{summary}
\vspace{10pt}

\newpage
%%========================================================%%
%%==============                                                   =============%%
%%======                                  Section1                               ======%%
%%====                                                                                    ====%%
%%==                                                                                            ==%%
%%====                                   Introduction                                 ====%%
%%======                                                                              ======%%
%%==============                                                  ==============%%
%%========================================================%%

\section{Introduction} \label{Sec1}

The {\it cancer} is one of the important study themes 
for human health.
In recent years, to analyze properties of the cancer mathematically, 
partial differential equations describing 
related biological phenomena %related to the cancer 
have been studied intensively. 
In this paper we focus on a {\it tumor angiogenesis}.
%In order to describe the branching of capillary sprout during 
%tumor angiogenesis according to chemotaxis, 
%Orme and Chaplain~\cite{OC-1996} proposed 
%the prototype of tumor angiogenesis models,
%The theme of this paper is to provide a mathematical analysis of 
%a tumor angiogenesis model with sensitivity functions. 
%\blue{(We would like to add one line here!)}
The prototype of chemotaxis-convection models of capillary-sprout growth during tumor angiogenesis, 
\begin{align*}
	\begin{cases}
		%==========  equations  ==========
		u_t=\Delta u-\nabla \cdot (\chi_1u\nabla v)+\nabla \cdot (\chi_2u\nabla w),
		\\[1.05mm]
		v_t=\Delta v+\nabla \cdot (\xi v\nabla w)+\alpha u-\beta v,
		\\[1.05mm]
		\tau w_t=\Delta w+\gamma u-\delta w, 
	\end{cases}
\end{align*}
where $\chi_1, \chi_2, \xi, \alpha, \beta, \gamma, \delta>0$ and 
$\tau \in \{0,1\}$ are constants, 
was proposed by Orme and Chaplain~\cite{OC-1996} 
%as the following system with $\chi_1=\chi_2=\xi=\tau=1$: 
to describe the branching of capillary sprout during tumor angiogenesis 
according to chemotaxis. 
Here {\it chemotaxis} is the property of cells to move in a directional manner 
in response to concentration gradients of chemical substances. 
%More precisely, this system represents that the endothelial cells secrete both matrix and fibronectin, 
%that the endothelial cells as well as the adhesive sites are carried with the extracellular matrix, 
%and that the endothelial cells move chemotactically toward increasing fibronectin concentrations. 
In this system the unknown functions $u, v, w$ idealize the density of endothelial cells, 
the concentration of a certain adhesive chemical, 
the distribution of the so-called extracellular matrix, respectively. 
Recently, important mathematical insights 
were given in 
\cite{BZ-2021, JX-2023, LT-2020, R-2021, SL-2021, TW-2021, ZK-2023}.  
Li and Tao~\cite{LT-2020} showed global existence and boundedness in the one-dimensional setting. 
After that,
Ren~\cite{R-2021} derived 
small data global existence and stabilization in the two-dimensional setting. 
Also, Tao and Winkler~\cite{TW-2021} established
global existence and boundedness 
%by making use of the effect of 
%the repulsion term 
when $\tau=0$. 
%Moreover, Bai and Zhang~\cite{BZ-2021} 
%constructed global generalized solutions of the system 
%in the two-dimensional setting when $\tau=0$. 
More related works %about the system with logistic term 
can be found in \cite{BZ-2021, JX-2023, SL-2021, ZK-2023}. 
%and an J\"{a}ger--Luckhaus system 
%can be referred in \cite{}. 
\medskip 

In this paper we are particularly interested in 
the case that $\chi_1, \chi_2, \xi$ are functions.
%At the cost of their largeness, 
%we would like to develop a mathematical analysis in general dimensions with $\tau=1$.
%Our view of this is based on the following two points of view. 
We would like to develop a mathematical analysis 
of the system with $\tau = 1$ in general dimensions according to the following two points of view. 
\begin{itemize}
\setlength{\leftskip}{0.5mm}
   \item Modelling in chemotaxis systems (the above system with $\chi_2=\xi=0$). 
   \vspace{-2mm}
   \item Mathematical analysis in attraction-repulsion systems 
            (the system with $\xi=0$). 
\end{itemize}
In the first point of view, 
Keller and Segel~\cite{KS-1970,KS-1971} proposed
chemotaxis systems which describe concentration phenomena 
such as in populations of cells; %caused by their property called \blue{chemotaxis; 
they first introduced the systems when $\chi_1$ is a constant in \cite{KS-1970}; 
after that, in \cite{KS-1971}, they proposed the systems when $\chi_1$ is a {\it sensitivity function}
to take into account the Weber--Fechner law. 
%Here chemotaxis is a property of cells to move in a directional manner 
%in response to concentration gradients of chemical substances. 
%In the studies on chemotaxis systems there
%are two cases that $\chi_1$ is a constant 
%and that $\chi_1$ is a function which is called 
%a {\it sensitivity function}.
%both of the constant sensitivity case and the signal-dependent sensitivity case. 
%Many types of chemotaxis systems have been studied 
%(see e.g.\ Osaki--Yagi~\cite{OY-2001}, Bellomo et al.\ \cite{BBTW-2015}, Arumugam--Tyagi~\cite{AT-2021}). 
%In another direction, to describe the aggregation of microglial cells in Alzheimer's disease 
%Luca~et~al.~\cite{LCEM-2003} proposed attraction-repulsion chemotaxis systems. 
%Also, these were introduced in~\cite[Section~3.3]{PH-2002} in order to represent 
%the quorum sensing effect that cells keep away from repulsive chemical substances. 
%After that, 
%As will be mentioned later, there are several mathematical studies on this system, 
%but they are restricted to those that do not include signal-dependent sensitivity functions. 
%One of the latter cases 
%%Here the signal-dependent sensitivity functions 
%comes from the Weber--Fechner law of stimulus perception in the process of chemotactic response, 
%which represents the situation that as the chemical stimulus increases, the chemotactic effect decreases. 
%\red{Chemotaxis systems with sensitivity functions were proposed in \cite{KS-1971} in order to 
%incorporate the Weber--Fechner law. 
Moreover, Painter and Hillen~\cite[p.\ 507]{PH-2002} 
explained the significance of sensitivity functions
in view of the receptor response law as follows: 

{\it 
\begin{quote}
If binding and disassociation occur on a much faster timescale than 
subsequent signal transduction, and the total number of cell surface receptors, 
$R_0$, remains about constant, then the number of activated receptors is 
$R_a = R_0v/(K +v)$, $K = k_-/k_+$. 
Assuming the chemotactic
response relates to the number of occupied receptors, $\tau(v) \propto R_a$ then
we obtain a receptor response law, [\textbf{\textit{55}}]:
\[
\chi(v)=\frac{\rho}{(K+v)^2}
\]
where $\rho$ is a constant.
\end{quote}
}

\noindent
In this sense it will be meaningful to study chemotaxis systems with sensitivity functions 
from biological points of view. 
In the second point from mathematical view, 
when $\chi_1, \chi_2$ are functions and $\xi=0$,  
global existence and boundedness were proved in \cite{CMY-2020} 
by a method using a test function defined as a combination of 
an exponential function and integrals of $\chi_1, \chi_2$. 
Also, as related works by utilizing a test function we can refer 
Fujie~\cite{F-2015}, 
Fujie and Senba~\cite{FS-2016-2,FS-2016-1}, 
Ahn~\cite{A-2019}, 
Frassu and Viglialoro~\cite{FV-2022}, 
Frassu, Rodr\'{\i}guez Galv\'{a}n and Viglialoro~\cite{FRV-2023}. 
Thus we expect that testing methods can be 
applied to a tumor angiogenesis model 
with sensitivity functions.%
\medskip

Based on the above observations, the purpose of this paper is to establish 
global existence and boundedness in the problem,
% 
%=====================  problem  =======================
    \begin{align} \label{ARC}
        \begin{cases}
        %==========  equations  ==========
        u_t=\Delta u-\nabla \cdot (u\chi_1(v)\nabla v)
                         +\nabla \cdot (u\chi_2(w)\nabla w),
        &x \in \Omega,\ t>0,
    \\[1.05mm]
        v_t=\Delta v+\nabla \cdot (v\xi(w)\nabla w)+\alpha u-\beta v,
        &x \in \Omega,\ t>0,
    \\[1.05mm]
        w_t=\Delta w+\gamma u-\delta w,
        &x \in \Omega,\ t>0,
    \\[1.8mm]
        %==========  boundary conditions  ==========
        \nabla u \cdot \nu=\nabla v \cdot \nu=\nabla w \cdot \nu=0, 
        &x \in \pa \Omega,\ t>0,
    \\[1.05mm]
        %==========  initial conditions  ==========
        u(x, 0)=u_0(x),\ v(x, 0)=v_0(x),\ w(x, 0)=w_0(x),
        &x \in \Omega,
        \end{cases}
    \end{align}
where $\Omega \subset \Rn$ $(n \le 3)$ is a bounded domain 
with smooth boundary $\pa \Omega$; 
$\nu$ is the outward normal vector to $\pa \Omega$; 
$\alpha, \beta, \gamma, \delta>0$ are constants; 
%generalizing the prototypes
%\begin{align*}
% \chi(s) = \frac{a}{(b+s)^k}, \quad \xi(s) = \frac{c}{(d+s)^{\ell}} 
% \quad (a,c>0,\ b, d \ge 0,\ k,\ell>1); 
%\end{align*}
$u, v, w$ are unknown 
functions; $u_0, v_0, w_0$ are initial data fulfilling
% 
%================  condition for known functions  ==================
    \begin{align} \label{KF}
            &u_0 \in C^0(\cl{\Omega}) \setminus\{0\}
    \quad
            {\rm and}
    \quad 
            u_0 \ge0\ {\rm in}\ \Omega,
\\[2mm] \label{KF2}
            &v_0, w_0 \in W^{1,\infty}(\Omega)
    \quad 
            {\rm and}
    \quad
               v_0, w_0 \ge 0\ {\rm in}\ \Omega.
%\quad
%            {\rm and}
%    \quad 
%   v_0>0,\ w_0>0.
    \end{align}
Also, $\chi_1, \chi_2$ are positive known 
functions satisfying
% 
%================  condition for \chi, \xi  ==================
\begin{align}
	\label{chi1class}
	&\chi_i \in C^{1+\vartheta_i}([0,\infty))\ 
	(0<\exists\, \vartheta_i<1),
	\quad 
	\chi_i>0\quad (i \in \{1,2\}),
	\\[2mm]
%	\label{chi2class}
%	&\chi_2 \in C^{1+\vartheta_2}([0,\infty))\ 
%	(0<\exists\, \vartheta_2<1),
%	\quad 
%	\chi_2>0,  
%	\\%[3.05mm]
%	\label{xiclass}
%	&\xi \in C^{1+\vartheta_3}([0,\infty))\ 
%	(0<\exists\, \vartheta_3<1),
%	\quad 
%	\xi>0,
%	\\[3.05mm]
	\label{chi1}
	&\sup_{s>0}s\chi_i(s)<\infty\quad (i \in \{1,2\}),	
	\\[2mm]
%	\label{chi2}
%	&\sup_{s>0}s\chi_2(s)<\infty,
%	\\%[3.05mm]
%	\label{xi}
%	&
%	\exists\,\xi_0>0;\quad \xi(s)\le\xi_0\chi_2(s)\quad 
%	{\rm for\ all}\ s \ge 0, 
%	\\
	\label{dchi1}
	&
	\exists\,K_i>0;\quad 
	\chi_i'(s)+K_i|\chi_i(s)|^2\le0\quad 
	{\rm for\ all}\ s \ge 0\quad (i \in \{1,2\})
%	\\
%	\label{dchi2}
%	&
%	\exists\,K_2>0;\quad 
%	\chi_2'(s)+K_2|\chi_2(s)|^2\le0\quad 
%	{\rm for\ all}\ s \ge 0,
%	\\
%	\label{dxi}
%	&
%	\exists\,K_3>0;\quad 
%	\xi'(s)\le K_3\quad 
%	{\rm for\ all}\ s \ge 0 \quad (i \in \{1,2\}). 
\end{align}
and $\xi$ is a positive known function fulfilling 
\begin{align}
%	\label{chi1class}
%	&\chi_i \in C^{1+\vartheta_i}([0,\infty))\ 
%	(0<\exists\, \vartheta_i<1),
%	\quad 
%	\chi_i>0\quad (i \in \{1,2\}),
%	\\%[3.05mm]
	%	\label{chi2class}
	%	&\chi_2 \in C^{1+\vartheta_2}([0,\infty))\ 
	%	(0<\exists\, \vartheta_2<1),
	%	\quad 
	%	\chi_2>0,  
	%	\\%[3.05mm]
	\label{xiclass}
	&\xi \in C^{1+\vartheta_3}([0,\infty))\ 
	(0<\exists\, \vartheta_3<1),
	\quad 
	\xi>0,
	\\[2mm]
%	\label{chi1}
%	&\sup_{s>0}s\chi_i(s)<\infty\quad (i \in \{1,2\}),	
%	\\%[3.05mm]
	%	\label{chi2}
	%	&\sup_{s>0}s\chi_2(s)<\infty,
	%	\\%[3.05mm]
	\label{xi}
	&
	\exists\,\xi_0>0;\quad \xi(s)\le\xi_0\chi_2(s)\quad 
	{\rm for\ all}\ s \ge 0, 
	\\[2mm]
%	\label{dchi1}
%	&
%	\exists\,K_i>0;\quad 
%	\chi_i'(s)+K_i|\chi_i(s)|^2\le0\quad 
%	{\rm for\ all}\ s \ge 0\quad (i \in \{1,2\}),
%	\\
	%	\label{dchi2}
	%	&
	%	\exists\,K_2>0;\quad 
	%	\chi_2'(s)+K_2|\chi_2(s)|^2\le0\quad 
	%	{\rm for\ all}\ s \ge 0,
	%	\\
	\label{dxi}
	&
	\exists\,K_3>0;\quad 
	\xi'(s)\le K_3\quad 
	{\rm for\ all}\ s \ge 0.
\end{align}
The example of $\chi_1, \chi_2, \xi$ are as follows: 
\begin{align*}
	\chi_1(s)=\frac{a_1}{(b_1+s)^{k_1}},\quad 
	\chi_2(s)=\frac{a_2}{(b_2+s)^{k_2}},\quad 
	\xi(s)=\frac{a_3}{(b_3+s)^{k_3}}
	\quad {\rm for}\ s \ge0
\end{align*}
with $a_1, a_2, a_3>0$, $b_1, b_2, b_3 \ge 0$ and $k_1, k_2, k_3>1$. 
In this example the constants $K_1, K_2$ 
in \eqref{dchi1} are given by $\frac{b_1k_1}{a_1}, \frac{b_2k_2}{a_2}$, respectively. 
Therefore some largeness conditions for $K_1, K_2$ 
imply that $a_1, a_2$ are small. 
\medskip

\noindent
{\bf Main result.} 
We obtain the following theorem which guarantees 
global existence and boundedness in \eqref{ARC} 
under some smallness conditions for $\chi_1, \chi_2, \xi$. 
% 
%=====================  Theorem 1.1  =======================
\begin{thm} \label{mainthm}
 Let\/ $\Omega \subset \Rn$ $(n \le 3)$ be a bounded domain 
 with smooth boundary. %and let $\mu>0$. 
 Assume that $\chi_1, \chi_2$ 
 satisfy \eqref{chi1class}--\eqref{dchi1} with $K_1, K_2$ fulfilling
 % 
 %==========  condition for \alpha, \beta  ==========
 \begin{align} \label{condipara}
 	K_1
 	>\min_{\lam\in J}\frac{2(2\lam+1)(3\lam+4)+\sqrt{D_\lam}}
 	{2(K_2-4)-\lam^2-2},
 	\quad 
 	K_2>4+\sqrt{2},
 \end{align}
 where 
 $J:= 
 \left(K_2-4-\sqrt{(K_2-4)^2-2},\, 
 K_2-4+\sqrt{(K_2-4)^2-2}\,\right)$
 and where
 % 
 %==========  def of D  ==========
 \begin{align*}
 	D_\lam
 	:= 4\lam\Big(2\left(K_2-4\right)+3\lam+8\Big)
 	\Big(\lam\left(K_2-4\right)+4\lam^2+12\lam+7\Big).
 \end{align*}
 Then there exists a constant $\xi_0^*>0$ such that 
 for all $\xi$ satisfying \eqref{xiclass}, \eqref{xi} with some
 $\xi_0 \in (0, \xi_0^*)$ as well as \eqref{dxi}, and all 
 $(u_0, v_0, w_0)$ fulfilling \eqref{KF}, \eqref{KF2} there exists a unique triplet $(u, v, w)$ of nonnegative functions
% 
%==========  class of solution  ==========
    \begin{align*}
        &u \in C^0(\cl{\Omega} \times [0,\infty)) 
                 \cap C^{2,1}(\cl{\Omega} \times (0,\infty)),
    \\
        &v, w \in \bigcap_{q>n} C^0([0,\infty); W^{1,q}(\Omega)) 
                 \cap C^{2,1}(\cl{\Omega} \times (0,\infty)),
%        &w \in C^0(\cl{\Omega} \times [0,\infty)) 
%                 \cap C^{2,1}(\cl{\Omega} \times (0,\infty)) 
%                 \cap L^\infty(0,\infty; W^{1,\infty}(\Omega)),
    \end{align*}
 which solves \eqref{ARC} in the classical sense. 
 Also, the solution is bounded in the sense that
% 
%==========  uniformly-in-time  ==========
    \begin{align*}
           \|u(\cdot, t)\|_{L^\infty(\Omega)}
%         +\|v(\cdot, t)\|_{W^{1,\infty}(\Omega)}
%         +\|w(\cdot, t)\|_{W^{1,\infty}(\Omega)} 
    \le C
    \end{align*}
 for all $t>0$ with some $C>0$.
\end{thm}

The strategy in the proof of Theorem~\ref{mainthm} 
is to derive $L^2$-boundedness of $u$ and
$L^r$-boundedness of $\nabla v$ with some $r>n$. 
The former can be shown by 
constructing the differential inequality
\begin{align*}
	\frac{d}{dt}\int_\Omega u^2(x,t)f(x,t)\,dx
	\le c_1\int_\Omega u^2(x,t)f(x,t)\,dx
	-c_2\Big(\int_\Omega u^2(x,t)f(x,t)\,dx\Big)^{1+\theta}
\end{align*}
for some constants $c_1, c_2, \theta>0$ and some test function 
$f$ (Section~\ref{Sec3}). 
%energy estimates. 
%Also, in the case $\xi=0$ the latter can be easily shown 
%by semigroup estimates. 
Once we obtained $L^2$-boundedness of $u$, 
the next step is to prove $L^r$-boundedness of $\nabla v$ 
with some $r>n$. 
Here in the case that $\xi=0$ it can be shown by semigroup estimates; 
however, in the case that $\xi \neq 0$ this method breaks down 
due to the presence of the term $\nabla \cdot (v\xi(w)\nabla w)$. 
On the other hand, the method for the case that $\xi$ is a constant  
in \cite{TW-2021} does not work since 
new terms which are difficult to handle 
(e.g.\ $|\nabla v|^{2r-2}v\xi'(w)\Delta w\nabla v \cdot \nabla w$)
appear. 
%$|\nabla v|^{2r-2}v\xi(w)\nabla v\cdot \nabla\Delta w$). }%
%In order to overcome this difficulty we first derive 
%%$L^2$-boundedness of $\nabla v$. 
%$L^2$-boundedness of $u$ (Section~\ref{Sec3}). 
%The key is to construct the differential inequality
%\begin{align*}
%	\frac{d}{dt}\int_\Omega u^2(x,t)f(x,t)\,dx
%	\le c_1\int_\Omega u^2(x,t)f(x,t)\,dx
%	    -c_2\Big(\int_\Omega u^2(x,t)f(x,t)\,dx\Big)^{1+\theta}
%\end{align*}
%for some $c_1, c_2, \theta>0$ and some function $f \colon \R^2 \to \R$. 
In order to overcome this difficulty we first observe 
$L^2$-boundedness of $\nabla v$ by an energy estimate 
%via some estimates for $w$ 
(Section~\ref{Sec4}). 
We next upgrade $L^2$-boundedness of $\nabla v$ to
$L^r$-boundedness of $\nabla v$ with some $r>n$ (Section~\ref{Sec5}). 
%\red{To see this we again need to estimate the term $\nabla \cdot (v\xi(w)\nabla w)$. 
%%In particular, since we can obtain an $L^\infty$-estimate for $v$ by semigroup estimates, 
%%we must deal with $\Delta w$. 
%}%
%This boundedness is asserted by using semigroup estimates, 
%but if it is used as is, the integral $\int_0^t (t-s)^{-1}\,ds$ appears and the integrability fails. 
%The key is to obtain an $L^2$-estimate for 
%$(-\Delta+\delta)^\rho u$ 
%with some $\rho \in (0,\frac12)$. 
%Thanks to this estimate, we can show $L^2$-boundedness of $\Delta w$, 
%which implies $L^r$-boundedness of $\nabla v$ with some $r>n$. 
Finally, in light of $L^r$-boundedness of $\nabla v$, we can obtain 
$L^\infty$-boundedness of $u$ (Section~\ref{Sec6}), 
which yields global existence and boundedness. 
%\medskip
%
%This paper is organized as follows. 
%In Section~\ref{Sec2} we collect some preliminary facts about
%local existence in \eqref{ARC}, 
%and a lemma which asserts $L^1$-boundedness of $v$ and $w$. 
%Sections~\ref{Sec3} and \ref{Sec4} are devoted to deriving 
%$L^2$-boundedness of $u$ and $\nabla v$, respectively. 
%In Section~\ref{Sec5} we obtain $L^r$-boundedness of $\nabla v$ with some $r>n$ 
%via some estimates for $u$ and $\Delta w$. 
%Finally, in Section~\ref{Sec6} we prove Theorem~\ref{mainthm}. 

%%========================================================%%
%%==============                                                  ==============%%
%%======                              Section2                                   ======%%
%%====                                                                                     ====%%
%%==                                                                                            ==%%
%%====                               Preliminaries                                    ====%%
%%======                                                                              ======%%
%%==============                                                  ==============%%
%%========================================================%%

\section{Preliminaries} \label{Sec2}

%\red{In this section we will collect some 
%lemmas which will be used later. }%
We first give local existence in \eqref{ARC}, 
which can be proved by standard arguments based on 
the construction mapping principle (see e.g.\ \cite{W-2010-2}). 
%
%%=====================  Lemma 2.1  =======================
\begin{lem} \label{LSE}
 Let\/ $\Omega \subset \Rn$ $(n \ge 1)$ be a bounded domain 
 	with smooth boundary and let 
 	$(u_0, v_0, w_0)$ fulfill \eqref{KF}, \eqref{KF2}. 
 Assume that $\chi_1, \chi_2, \xi$ satisfy 
 \eqref{chi1class}, \eqref{xiclass}. 
 Then there exists $\Tmax \in (0,\infty]$ such that 
 \eqref{ARC} admits a unique classical solution 
 $(u, v, w)$ such that
% 
%==========  class of solution  ==========
    \begin{align*}
        &u \in C^0(\cl{\Omega} \times [0,\tmax)) 
                 \cap C^{2,1}(\cl{\Omega} \times (0,\tmax)),
    \\
        &v, w \in \bigcap_{q>n} C^0([0,\tmax); W^{1,q}(\Omega)) 
                 \cap C^{2,1}(\cl{\Omega} \times (0,\tmax)), 
%    \\
%        &w \in C^0(\cl{\Omega} \times [0,\Tmax)) 
%                 \cap C^{2,1}(\cl{\Omega} \times (0,\Tmax)) 
%                 \cap L^\infty_\mathrm{loc}([0,\Tmax); W^{1,\infty}(\Omega))
    \end{align*}
 and $u, v, w$ have positivity. 
 In addition, $u$ has the mass conservation property
 \begin{align*}
 \int_\Omega u(\cdot, t)=\int_\Omega u_0
 \end{align*}
 for all $t \in (0,\tmax)$, 
and moreover, 
% 
%==========  extension criterion  ==========
    \begin{align} \label{BU}
            {\it if}\ \Tmax<\infty,
    \quad 
            {\it then}\ \limsup_{t \nearrow \Tmax} 
                           \left(\|u(\cdot,t)\|_{L^\infty(\Omega)}
                           +\|v(\cdot,t)\|_{W^{1, \infty}(\Omega)}
                           +\|w(\cdot,t)\|_{W^{1, \infty}(\Omega)}\right)=\infty.
    \end{align}
\end{lem}
%%==================  Proof of Lemma 2.1  ====================
%\begin{proof}
% Using a standard argument based on the contraction mapping principle 
% as in~\cite[Lemma 3.1]{TW-2013}, we can show local existence and 
% blow-up criterion \eqref{BU}. 
%Note that the mass conservation property \eqref{mcp} can be obtained 
% by integrating the first equation in \eqref{ARC} 
% over $\Omega \times (0, t)$ for $t\in(0, \Tmax)$, and that the lower estimates \eqref{LEw} follow from Lemma~\ref{LB}. 
%\end{proof}

In the following we assume that $\Omega \subset \R^n$ 
$(n \le 3)$ 
is a bounded domain with smooth boundary and that 
$\chi_1, \chi_2, \xi$ fulfill \eqref{chi1class}, \eqref{xiclass} as well as that
%$\mu>0$ and 
$(u_0, v_0, w_0)$ satisfies \eqref{KF}, \eqref{KF2}. 
We then let the triplet $(u, v, w)$ denote the local classical
solution of \eqref{ARC} given in Lemma~\ref{LSE} and 
$\Tmax$ denote its maximal existence time. 
\medskip

We next state a lemma which guarantees $L^1$-boundedness of $v, w$. 

\begin{lem}\label{L1vw}
   Let\/ $\Omega \subset \Rn$ $(n \ge 1)$ be a bounded domain 
  	with smooth boundary and let
  	$(u_0, v_0, w_0)$ fulfill \eqref{KF}, \eqref{KF2}. 
  Assume that $\chi_1, \chi_2, \xi$ satisfy 
  \eqref{chi1class}, \eqref{xiclass}. 
  Then there exists $C>0$ such that 
 \begin{align*}
 	\int_\Omega v(\cdot,t) \le C,\qquad \int_\Omega w(\cdot,t)\le C
 \end{align*}
for all $t \in (0,\tmax)$. 
\end{lem}

\begin{proof}
  Integrating the second equation in \eqref{ARC} over $\Omega$ and 
  noting the Neumann boundary condition for $v$ in \eqref{ARC} yield
  \begin{align*}
  	\frac{d}{dt}\int_\Omega v \le -\beta \int_\Omega v+c_1
  \end{align*}
 with some $c_1>0$, which implies that 
 $L^1$-boundedness of $v$ holds. 
 Similarly, we can also verify $L^1$-boundedness of $w$. 
\end{proof}

\section{$L^2$-boundedness of $u$} \label{Sec3}

In this section we will show $L^2$-boundedness of $u$. 
%In the case $\xi=0$, by semigroup estimates, an $L^p$-estimate for $u$ with some $p>\frac{n}{2}$ yields 
%a $W^{1,q}$-estimate for $v$ with some $q>n$, which implies an $L^\infty$-estimate for $u$. 
%However, In the case $\xi\neq0$, since semigroup estimates do not work, it is difficult to derive an estimate for $\nabla v$. 
%Let $r \in (n, \frac{2n}{(n-2)_+})$. 
%The cornerstone of the proof of Theorem~\ref{mainthm} is 
%to obtain $L^r$-boundedness of $\nabla v$. 
%To see this we need an estimate for $\Delta w$, 
%because the second equation in \eqref{ARC} includes the term $\nabla \cdot (v\xi(w)\nabla w)$. 
%%%\begin{align*}
%%	$\nabla \cdot (v\xi(w)\nabla w)
%%	=\xi(w)\nabla v \cdot \nabla w
%%	+v\xi'(w)|\nabla w|^2+v\xi(w)\Delta w$. 
%%%\end{align*}
%%However, attempting to derive an estimate for $\Delta w$ 
%%by semigroup estimates, we face an integral which is not integrable. 
%We will derive an estimate for $(-\Delta+\delta)^\rho u$ 
%with some $\rho \in (0,\frac12)$ 
%toward an estimate for $\Delta w$. 
%To obtain an estimate for $(-\Delta+\delta)^\rho u$ 
%So we start by $L^2$-estimate for $u$. 
We introduce the function $f=f(x,t)$ by
\begin{align*}
f(x,t):=\exp\left(-r\int_0^{v(x,t)} \chi_1(s)\,ds
-\sigma\int_0^{w(x,t)} \chi_2(s)\,ds\right),
\end{align*}
where $r,\sigma>0$ are some constants which will be fixed later. 

\begin{lem}
Let $r, \sigma>0$. 
Then for all $t \in (0, \tmax)$, 
\begin{align}\label{du2f}
   \frac{d}{dt}\int_\Omega u^2f
=I_1+I_2+I_3-r\int_\Omega u^2f\chi_1(v)(\alpha u-\beta v)-\sigma\int_\Omega u^2f\chi_2(w)(\gamma u-\delta w),
\end{align}
where
\begin{align*}
I_1 &:= 2\int_\Omega uf\nabla \cdot\Big(\nabla u-u\chi_1(v)\nabla v+u\chi_2(w)\nabla w\Big),\\
I_2 &:= -r\int_\Omega u^2f\chi_1(v)\nabla \cdot \Big(\nabla v+v\xi(w)\nabla w\Big),\\
I_3 &:= -\sigma\int_\Omega u^2f\chi_2(w)\Delta w. 
\end{align*}
\end{lem}

\begin{proof}
By virtue of the definition of $f$, we obtain 
\begin{align*}
   \frac{d}{dt}\int_\Omega u^2f
= 2\int_\Omega u_t uf-\int_\Omega u^2 f \cdot \Big(r\chi_1(v)v_t+\sigma\chi_2(w)w_t\Big), 
\end{align*}
which in conjunction with the equations in \eqref{ARC} implies
\begin{align*}
   \frac{d}{dt}\int_\Omega u^2f
&= 2\int_\Omega uf\cdot
\Big(\Delta u-\nabla \cdot (u\chi_1(v)\nabla v)+\nabla \cdot (u\chi_2(w)\nabla w)\Big)\\
&\quad\,-r\int_\Omega u^2f\chi_1(v)\Big(\Delta v+\nabla \cdot (v\xi(w)\nabla w)+\alpha u-\beta v\Big)\\
&\quad\,-\sigma\int_\Omega u^2f\chi_2(w)\Big(\Delta w+\gamma u-\delta w\Big)\\
&= 2\int_\Omega uf\nabla \cdot\Big(\nabla u-u\chi_1(v)\nabla v+u\chi_2(w)\nabla w\Big)\\
&\quad\,-r\int_\Omega u^2f\chi_1(v)\nabla \cdot \Big(\nabla v+v\xi(w)\nabla w\Big)-\sigma\int_\Omega u^2f\chi_2(w)\Delta w\\
&\quad\,-r\int_\Omega u^2f\chi_1(v)(\alpha u-\beta v)-\sigma\int_\Omega u^2f\chi_2(w)(\gamma u-\delta w), 
\end{align*}
which concludes the proof. 
\end{proof}

The key to the derivation of $L^2$-boundedness of $u$ is 
to obtain an estimate for $I_1+I_2+I_3$.
%by an integral including $|\nabla u|^2$. 
%To see this we estimate $I_1+I_2+I_3$. 

\begin{lem}\label{LemI1I2I3}
Assume that $\chi_1, \chi_2, \xi$ satisfy \eqref{chi1class}--\/\eqref{dxi} 
 with  $K_1, K_2$ fulfilling \eqref{condipara}. 
Let $r, \sigma>0$, $\ep \in [0,1)$ and put 
\begin{align*}
\textsf{x}:=u^{-1}|\nabla u|,\quad \textsf{y}:=\chi_1(v)|\nabla v|,\quad \textsf{z}:=\chi_2(w)|\nabla w|.
\end{align*}
Then for all $t \in (0, \tmax)$, 
\begin{align}\label{I_1+I_2+I_3}
I_1+I_2+I_3 
&\le -2\ep\int_\Omega f|\nabla u|^2\notag\\
&\quad\,+\int_\Omega u^2f \cdot(-a_1(\ep)\textsf{x}^2+a_2\textsf{x}\textsf{y}+a_3(\xi_0)\textsf{x}\textsf{z}
                                                  -a_4\textsf{y}^2+a_5(\xi_0)\textsf{y}\textsf{z}-a_6\textsf{z}^2),
\end{align}
where
\begin{align*}
a_1(\ep)&:=2(1-\ep),\\
a_2&:=2(2r+1),\\
a_3(\xi_0)&:=2(2\sigma+1+b\xi_0r),\\
a_4&:=r(r+K_1+2),\\ 
a_5(\xi_0)&:=2(r\sigma+r+\sigma)+b\xi_0r(r+K_1),\\
a_6&:=\sigma(\sigma+K_2-2)
\end{align*}
with some $b>0$. 
\end{lem}

\begin{proof}
Integration by parts and the definition of $f$ yield
\begin{align}\label{I_1}
I_1
&=-2\int_\Omega \nabla (uf)\cdot\Big(\nabla u-u\chi_1(v)\nabla v+u\chi_2(w)\nabla w\Big)\notag\\
&=-2\int_\Omega f\nabla u\cdot\Big(\nabla u-u\chi_1(v)\nabla v+u\chi_2(w)\nabla w\Big)\notag\\
&\quad\,+2\int_\Omega uf\cdot\Big(r\chi_1(v)\nabla v+\sigma\chi_2(w)\nabla w\Big)\cdot 
                                              \Big(\nabla u-u\chi_1(v)\nabla v+u\chi_2(w)\nabla w\Big)\notag\\
&=-2\int_\Omega f|\nabla u|^2
    +2\int_\Omega uf\chi_1(v)\nabla u \cdot \nabla v
    -2\int_\Omega uf\chi_2(w)\nabla u \cdot \nabla w\notag\\
&\quad\,+2r\int_\Omega uf\chi_1(v)\nabla u \cdot \nabla v
    +2\sigma\int_\Omega uf\chi_2(w)\nabla u \cdot \nabla w\notag\\
&\quad\,-2r\int_\Omega u^2f|\chi_1(v)|^2|\nabla v|^2
    +2\sigma\int_\Omega u^2f|\chi_2(w)|^2|\nabla w|^2\notag\\
&\quad\,+2(r-\sigma)\int_\Omega u^2f\chi_1(v)\chi_2(w)\nabla v \cdot \nabla w,\notag\\
\intertext{
which along with the Cauchy--Schwarz inequality implies
}
I_1&\le -2\int_\Omega f|\nabla u|^2
    +2(r+1)\int_\Omega uf\chi_1(v)|\nabla u||\nabla v|
    +2(\sigma+1)\int_\Omega uf\chi_2(w)|\nabla u||\nabla w|\notag\\
&\quad\,-2r\int_\Omega u^2f|\chi_1(v)|^2|\nabla v|^2
    +2(r+\sigma)\int_\Omega u^2f\chi_1(v)\chi_2(w)|\nabla v||\nabla w|\notag\\
&\quad\,+2\sigma\int_\Omega u^2f|\chi_2(w)|^2|\nabla w|^2.
\end{align}
Also, again by the definition of $f$, we derive 
\begin{align*}
I_2
&=r\int_\Omega \nabla \Big(u^2f\chi_1(v)\Big) \cdot \Big(\nabla v+v\xi(w)\nabla w\Big)\\
&=2r\int_\Omega uf\chi_1(v)\nabla u \cdot \Big(\nabla v+v\xi(w)\nabla w\Big)\\
&\quad\,-r\int_\Omega u^2f\chi_1(v)\cdot\Big(r\chi_1(v)\nabla v+\sigma\chi_2(w)\nabla w\Big)\cdot  \Big(\nabla v+v\xi(w)\nabla w\Big)\\
&\quad\,+r\int_\Omega u^2f\chi_1'(v) \nabla v \cdot \Big(\nabla v+v\xi(w)\nabla w\Big)\\
&=2r\int_\Omega uf\chi_1(v)\nabla u \cdot \nabla v
    +2r\int_\Omega uvf\chi_1(v)\xi(w)\nabla u \cdot \nabla w\\
&\quad\,-r^2\int_\Omega u^2f|\chi_1(v)|^2|\nabla v|^2
    -r\sigma\int_\Omega u^2f\chi_1(v)\chi_2(w)\nabla v\cdot\nabla w\\
&\quad\,-r^2\int_\Omega u^2vf|\chi_1(v)|^2\xi(w)\nabla v \cdot \nabla w
    -r\sigma\int_\Omega u^2vf\chi_1(v)\chi_2(w)\xi(w)|\nabla w|^2\\
&\quad\,+r\int_\Omega u^2f\chi_1'(v)|\nabla v|^2
    +r\int_\Omega u^2vf\chi_1'(v)\xi(w)\nabla v \cdot \nabla w.
\end{align*}
Here, noting from \eqref{chi1} that $v\chi_1(v) \le c_1$ with some $c_1>0$ 
and using the Cauchy--Schwarz inequality, 
we obtain from the conditions \eqref{dchi1}, \eqref{xi} that
\begin{align}\label{I_2}
I_2
&\le 2r\int_\Omega uf\chi_1(v)|\nabla u||\nabla v|
    +2c_1\xi_0r\int_\Omega uf\chi_2(w)|\nabla u||\nabla w|\notag\\
&\quad\,-r^2\int_\Omega u^2f|\chi_1(v)|^2|\nabla v|^2
    +r\sigma\int_\Omega u^2f\chi_1(v)\chi_2(w)|\nabla v||\nabla w|\notag\\
&\quad\,+c_1\xi_0r^2\int_\Omega u^2f\chi_1(v)\chi_2(w)|\nabla v||\nabla w|\notag\\
&\quad\,-rK_1\int_\Omega u^2f|\chi_1(v)|^2|\nabla v|^2
    +c_1K_1\xi_0r\int_\Omega u^2f\chi_1(v)\chi_2(w)|\nabla v||\nabla w|\notag\\
&=2r\int_\Omega uf\chi_1(v)|\nabla u||\nabla v|
     -(r^2+rK_1)\int_\Omega u^2f|\chi_1(v)|^2|\nabla v|^2\notag\\
&\quad\,+r\sigma\int_\Omega u^2f\chi_1(v)\chi_2(w)|\nabla v||\nabla w|\notag\\
&\quad\,+2c_1\xi_0r\int_\Omega uf\chi_2(w)|\nabla u||\nabla w|
     +c_1\xi_0r(r+K_1)\int_\Omega u^2f\chi_1(v)\chi_2(w)|\nabla v||\nabla w|. 
\end{align}
Similarly, we have from the Cauchy--Schwarz inequality 
and the condition \eqref{dchi1} that 
\begin{align}\label{I_3}
I_3
&=\sigma\int_\Omega \nabla \Big(u^2f\chi_2(w)\Big) \cdot \nabla w\notag\\
&=2\sigma\int_\Omega uf\chi_2(w)\nabla u \cdot \nabla w\notag\\
     &\quad\,-\sigma\int_\Omega u^2f\chi_2(w)\cdot
                  \Big(r\chi_1(v)\nabla v+\sigma\chi_2(w)\nabla w\Big)\cdot \nabla w\notag\\
     &\quad\,+\sigma\int_\Omega u^2f\chi_2'(w) |\nabla w|^2\notag\\
&\le2\sigma\int_\Omega uf\chi_2(w)|\nabla u||\nabla w|\notag\\
     &\quad\,+r\sigma\int_\Omega u^2f\chi_1(v)\chi_2(w)|\nabla v||\nabla w|
                  -\sigma^2\int_\Omega u^2f|\chi_2(w)|^2|\nabla w|^2\notag\\
     &\quad\,-\sigma K_2\int_\Omega u^2f|\chi_2(w)|^2 |\nabla w|^2\notag\\
&=2\sigma\int_\Omega uf\chi_2(w)|\nabla u||\nabla w|\notag\\
     &\quad\,+r\sigma\int_\Omega u^2f\chi_1(v)\chi_2(w)|\nabla v||\nabla w|
                   -(\sigma^2+\sigma K_2)\int_\Omega u^2f|\chi_2(w)|^2|\nabla w|^2.
\end{align}
Hence combining \eqref{I_1}--\eqref{I_3} implies 
\begin{align*}
&I_1+I_2+I_3\\
&\quad\,\le -2\int_\Omega f|\nabla u|^2
       +a_2\int_\Omega uf\chi_1(v)|\nabla u||\nabla v|
       +a_3(\xi_0)\int_\Omega uf\chi_2(w)|\nabla u||\nabla w|\\
&\qquad\ -a_4\int_\Omega u^2f\chi_1^2(v)|\nabla v|^2
       +a_5(\xi_0)\int_\Omega u^2f\chi_1(v)\chi_2(w)|\nabla v||\nabla w|
       -a_6\int_\Omega u^2f\chi_2^2(w)|\nabla w|^2\\
&\quad\,=\int_\Omega u^2f \cdot(-2\textsf{x}^2+a_2\textsf{x}\textsf{y}+a_3(\xi_0)\textsf{x}\textsf{z}
                                                  -a_4\textsf{y}^2+a_5(\xi_0)\textsf{y}\textsf{z}-a_6\textsf{z}^2),
\end{align*}
where
\begin{align*}
a_2&:=2(2r+1),\\
a_3(\xi_0)&:=2(2\sigma+1+c_1\xi_0r),\\
a_4&:=r(r+K_1+2),\\
a_5(\xi_0)&:=2(r\sigma+r+\sigma)+c_1\xi_0r(r+K_1),\\
a_6&:=\sigma(\sigma+K_2-2),
\end{align*}
and 
\begin{align*}
	\textsf{x}:=u^{-1}|\nabla u|,\quad
	\textsf{y}:=\chi_1(v)|\nabla v|,\quad
	\textsf{z}:=\chi_2(w)|\nabla w|.
\end{align*}
Thus, noting that $-2=-2(1-\ep)-2\ep=:-a_1(\ep)-2\ep$ for $\ep \in [0,1)$, we infer \eqref{I_1+I_2+I_3}. 
\end{proof}

We finally derive $L^2$-boundedness of $u$ 
via proving that there exist $r, \sigma>0$ such that 
the quadratic form
 $-a_1(\ep)\textsf{x}^2+a_2\textsf{x}\textsf{y}+a_3(\xi_0)\textsf{x}\textsf{z}
-a_4\textsf{y}^2+a_5(\xi_0)\textsf{y}\textsf{z}-a_6\textsf{z}^2$ is negative.

\begin{lem}\label{Lem:uL2}
Assume that $\chi_1, \chi_2, \xi$ satisfy \eqref{chi1class}--\/\eqref{dxi} 
with $K_1, K_2$ fulfilling \eqref{condipara}. 
 Then there exist $r, \sigma>0$ such that 
 \begin{align}\label{uL2}
  \|u(\cdot,t)\|_{L^2(\Omega)} \le C
 \end{align}
 for all $t \in (0, \tmax)$ with some $C>0$. 
\end{lem}

\begin{proof}
Let us show that $-a_1(\ep)\textsf{x}^2+a_2\textsf{x}\textsf{y}+a_3(\xi_0)\textsf{x}\textsf{z}
-a_4\textsf{y}^2+a_5(\xi_0)\textsf{y}\textsf{z}-a_6\textsf{z}^2 < 0$ with some $r, \sigma>0$ in \eqref{I_1+I_2+I_3}. 
We set
\begin{align*}% \label{Syl2}
	A_1(\ep,\xi_0):=\left|
	\begin{array}{cc}
		-a_1(\ep) & \frac{a_3(\xi_0)}{2}
		\\
		\frac{a_3(\xi_0)}{2} & -a_6
	\end{array}
	\right|
	\quad
	{\rm and}
	\quad
	A_2(\ep,\xi_0):=\left|
	\begin{array}{ccc}
		-a_1(\ep) & \frac{a_3(\xi_0)}{2} & \frac{a_2}{2}
		\\
		\frac{a_3(\xi_0)}{2} &            -a_6 & \frac{a_5(\xi_0)}{2}
		\\
		\frac{a_2}{2} & \frac{a_5(\xi_0)}{2} &            -a_4
	\end{array}
	\right|
\end{align*}
for $\ep \in [0, 1)$ and $\xi_0>0$. 
By the proof of \cite[Lemma~3.3]{CMY-2020}, 
we obtain that there exist $r, \sigma>0$ such that 
$A_1(0,0)>0$ and $A_2(0,0)<0$. 
Thus, noting that the functions
$a_1 \colon \ep \mapsto 2(1-\ep)$,  
$a_3 \colon \xi_0 \mapsto 2(2\sigma+1+b\xi_0r)$, 
$a_5 \colon \xi_0 \mapsto 2(r\sigma+r+\sigma)+b\xi_0r(r+K_1)$, 
where $b>0$ is obtained in Lemma~\ref{LemI1I2I3}, 
are continuous at $\ep=0$, $\xi_0=0$, 
we can find $\ep_0 \in (0,1)$ and $\xi_0^*>0$ 
such that $A_1(\ep,\xi_0)>0$ and $A_2(\ep,\xi_0)<0$ for all $\ep \in (0,\ep_0)$ 
and all $\xi_0 \in (0, \xi_0^*)$. 
Hence in light of these estimates and the Sylvester criterion (see e.g.\ \cite[Lemma~2.4]{CMY-2020}), 
we see that
\begin{align*}
	-a_1(\ep)\textsf{x}^2+a_2\textsf{x}\textsf{y}+a_3(\xi_0)\textsf{x}\textsf{z}
	-a_4\textsf{y}^2+a_5(\xi_0)\textsf{y}\textsf{z}-a_6\textsf{z}^2 < 0, 
\end{align*}
which along with \eqref{I_1+I_2+I_3} and \eqref{du2f} implies that 
\begin{align*}
	\frac{d}{dt}\int_\Omega u^2f
	+2\ep\int_\Omega f|\nabla u|^2
	&\le
	-r\int_\Omega u^2f\chi_1(v)(\alpha u-\beta v)
	-\sigma\int_\Omega u^2f\chi_2(w)(\gamma u-\delta w).
\end{align*}
%for all $t \in (0, \tmax)$. 
Therefore, proceeding as in \cite[the proof of Lemma~3.5]{CMY-2021} with $p=2$, $\chi=\chi_1$, $\xi=\chi_2$, 
we can arrive at the conclusion of this lemma. 
\end{proof}

\section{$L^2$-boundedness of $\nabla v$} \label{Sec4}

Throughout the sequel, we assume that $\chi_1, \chi_2$ satisfy \eqref{chi1class}--\eqref{dchi1} with $K_1, K_2$ fulfilling 
\eqref{condipara}
and $\xi$ satisfies \eqref{xiclass}, \eqref{xi} with some
$\xi_0 \in (0, \xi_0^*)$ as well as \eqref{dxi}. 
\medskip

To derive $L^r$-boundedness of $\nabla v$ with some $r>n$
we will obtain $L^2$-boundedness of $\nabla v$ 
and then upgrade it. 
%To see this we need to deal with the term $\nabla \cdot (v\xi(w)\nabla w)$ 
%via some estimates for $w$, 
%since the second equation in \eqref{ARC} includes this term. 
We first prove the following lemma which will be used later. 

\begin{lem}
%Assume that $\chi_1, \chi_2, \xi$ fulfill \eqref{chi1class}--\/\eqref{dxi} 
%with  $K_1, K_2$ which satisfy \eqref{condipara}. 
There exists $q \in (n, \frac{2n}{(n-2)_+})$ such that 
\begin{align}\label{wLq}
	\|\nabla w(\cdot,t)\|_{L^q(\Omega)} \le C_1
\end{align}
 for all $t \in (0, \tmax)$ with some $C_1>0$. 
 Moreover, 
\begin{align}\label{dwL2}
 \int_t^{t+\tau}\int_\Omega |\Delta w(\cdot,t)|^2 \le C_2
\end{align}
for all $t \in (0, \tmax-\tau)$ with some $C_2>0$, 
where $\tau:=\min\{1, \frac{1}{2}\tmax\}$. 
\end{lem}

\begin{proof}
We can obtain $L^q$-boundedness \eqref{wLq} by the estimate \eqref{uL2}
and semigroup estimates (see e.g.\ \cite[Lemma~1.3]{W-2010-1}). 
Thus we concentrate on the proof of \eqref{dwL2}. 
Integration by parts, 
the third equation in \eqref{ARC} and the Young inequality derive that 
for all $t \in (0, \tmax)$,
\begin{align*}
	      \frac{1}{2}\frac{d}{dt}\int_\Omega |\nabla w|^2
	&= -\int_\Omega \Delta w \cdot w_t
\\
    &= -\int_\Omega \Delta w \cdot (\Delta w+\gamma u-\delta w)\
\\
	&= -\int_\Omega |\Delta w|^2
	      -\gamma\int_\Omega u\Delta w
	      -\delta\int_\Omega |\nabla w|^2
\\
	&\le -\frac{1}{2}\int_\Omega |\Delta w|^2
	       -\delta\int_\Omega|\nabla w|^2
	       +\frac{\gamma^2}{2}\int_\Omega u^2.
\end{align*}
%for all $t \in (0, \tmax)$ 
This along with \eqref{uL2} yields
\begin{align}\label{Diffnabw}
	     \frac{d}{dt}\int_\Omega |\nabla w|^2
	     +2\delta\int_\Omega|\nabla w|^2
	     +\int_\Omega |\Delta w|^2
	\le c_1
\end{align}
%for all $t \in (0, \tmax)$ 
for all $t \in (0, \tmax)$ with some $c_1>0$, which implies
\begin{align}\label{dwL21}
	\|\nabla w(\cdot,t)\|_{L^2(\Omega)} \le c_2
\end{align}
%for all $t \in (0, \tmax)$ 
for all $t \in (0, \tmax)$ with some $c_2>0$. 
Thus, by integrating \eqref{Diffnabw} over $[t,t+\tau]$ for $t \in (0,\tmax-\tau)$, 
we see from \eqref{dwL21} that 
\begin{align*}
		 \int_\Omega |\nabla w(\cdot,t+\tau)|^2
		 +2\delta\int_t^{t+\tau}\int_\Omega|\nabla w|^2
		 +\int_t^{t+\tau}\int_\Omega |\Delta w|^2
	&\le c_1\tau+\int_\Omega |\nabla w(\cdot,t)|^2
	\\
	&\le c_1\tau+c_2^2
\end{align*}
for all $t \in (0,\tmax-\tau)$, 
which in conjunction with the nonnegativity of 
$|\nabla w|^2$ asserts that \eqref{dwL2} holds. 
\end{proof}

The next lemma can be proved by arguments based on semigroup estimates 
(see e.g.\ \cite[Proof of Lemma~3.2]{BBTW-2015}, \cite[Proof of Lemma~2.4]{M-2018}). 
Thus we shall only show brief proof. 

\begin{lem}\label{Lem:vLinf}
%Assume that $\chi_1, \chi_2, \xi$ fulfill \eqref{chi1class}--\/\eqref{dxi} 
%with  $K_1, K_2$ which satisfy \eqref{condipara}. 
There exists $C>0$ such that
\begin{align}\label{Linfestv}
\|v(\cdot,t)\|_{L^\infty(\Omega)} \le C
\end{align}
for all $t \in (0, \tmax)$. 
\end{lem}

\begin{proof}
%Let $N(T):=\sup_{t \in (0,T)} \|v(\cdot,t)\|_{L^\infty(\Omega)}$ for $T \in (0,\tmax)$. 
It is sufficient to prove that there is $c_1>0$ such that 
$\sup_{t \in (0,T)} \|v(\cdot,t)\|_{L^\infty(\Omega)} \le c_1$ 
for all $T \in (0,\tmax)$. 
To this end we put $t_0:=(t-1)_+$ for $t \in (0,T)$ and 
let the symbol $(e^{t\Delta})_{t\ge0}$ denote 
the Neumann heat semigroup in $\Omega$. 
Then we can rewrite the second equation in \eqref{ARC} by the Duhamel formula as
\begin{align*}
	v(\cdot,t)
	&= e^{(t-t_0)(\Delta-\beta)}v(\cdot,t_0)
	-\int_{t_0}^t e^{(t-s)(\Delta-\beta)}
	\nabla \cdot \Big(v(\cdot,s)\xi(w(\cdot,s))\nabla w(\cdot,s)\Big)\,ds\\
	&\quad\,
	+\alpha \int_{t_0}^t e^{(t-s)(\Delta-\beta)}u(\cdot,s)\,ds\\
	&=: J_1(\cdot,t)+J_2(\cdot,t)+J_3(\cdot,t)
\end{align*}
for all $t \in (0, T)$. 
We can estimate $J_1, J_3$ by constants 
via semigroup arguments. 
Also, by virtue of the fact $\xi(s) \le \xi(0)$ from the conditions 
\eqref{dchi1}, \eqref{xi}, 
and the estimate \eqref{wLq}, the term $J_2$ can be estimated as 
$\|J_2(\cdot,t)\|_{L^\infty(\Omega)} 
\le c_2\sup_{t \in (0,T)} \|v(\cdot,t)\|_{L^\infty(\Omega)}^{\theta}$ 
with some $c_2>0$ and $\theta \in (0,1)$. 
Thus we have
\begin{align*}
\sup_{t \in (0,T)} \|v(\cdot,t)\|_{L^\infty(\Omega)} 
\le c_2\sup_{t \in (0,T)} \|v(\cdot,t)\|_{L^\infty(\Omega)}^{\theta}+c_3
\end{align*} 
with some $c_3>0$, 
which yields the conclusion of this lemma. 
\end{proof}

We are now in the position to derive $L^2$-boundedness of $\nabla v$. 

\begin{lem}\label{Lem:dvL2}
%Assume that $\chi_1, \chi_2, \xi$ fulfill \eqref{chi1class}--\/\eqref{dxi} 
%with  $K_1, K_2$ which satisfy \eqref{condipara}. 
There exists $C>0$ such that
\begin{align}\label{dvL2est}
\|\nabla v(\cdot,t)\|_{L^2(\Omega)} \le C
\end{align}
for all $t \in (0, \tmax)$. 
\end{lem}

\begin{proof}
Integration by parts, 
the second equation in \eqref{ARC} 
and the Cauchy--Schwarz inequality imply that
\begin{align*}
	\frac{1}{2}\frac{d}{dt}\int_\Omega |\nabla v|^2
	&= -\int_\Omega \Delta v \cdot v_t
	\\
	&= -\int_\Omega \Delta v \cdot 
	               \Big(\Delta v+\nabla\cdot (v\xi(w)\nabla w)+\alpha u-\beta v\Big)
	\\
	&= -\int_\Omega |\Delta v|^2
	-\int_\Omega \Delta v \cdot 
	   \Big(\xi(w)\nabla v \cdot \nabla w+v\xi'(w)|\nabla w|^2+v\xi(w)\Delta w\Big)
    \\
    &\quad\,
          +\alpha\int_\Omega \nabla u \cdot \nabla v
          -\beta\int_\Omega |\nabla v|^2
	\\
	&\le -\int_\Omega |\Delta v|^2
	+\int_\Omega \xi(w)|\Delta v||\nabla v||\nabla w|
	+\int_\Omega v|\xi'(w)||\Delta v||\nabla w|^2
	\\
	&\quad\,
	+\int_\Omega v\xi(w)|\Delta v||\Delta w|
	-\alpha\int_\Omega u\Delta v
	-\beta\int_\Omega |\nabla v|^2.
\end{align*}
%for all $t \in (0,\tmax)$. 
Here, noting from \eqref{dchi1}, \eqref{xi} that 
$\xi(s) \le \xi(0)$, we see from \eqref{dxi}, \eqref{Linfestv} that 
\begin{align*}%\label{ddvL2}
	\frac{1}{2}\frac{d}{dt}\int_\Omega |\nabla v|^2
		&\le -\int_\Omega |\Delta v|^2
	+c_1\int_\Omega |\Delta v||\nabla v||\nabla w|
	+c_1\int_\Omega |\Delta v||\nabla w|^2 
	\notag\\
	&\quad\,
	+c_1\int_\Omega |\Delta v||\Delta w|
	-\alpha\int_\Omega u\Delta v
	-\beta\int_\Omega |\nabla v|^2
\end{align*}
%for all $t \in (0,\tmax)$ 
with some $c_1>0$. 
We now take $\ep_1>0$ sufficiently small and let $\eta>0$ which will be fixed later. 
Applying the Young inequality to 
the second, third, fourth and fifth terms on the right-hand side 
of this inequality, we have
\begin{align*}
	\frac{1}{2}\frac{d}{dt}\int_\Omega |\nabla v|^2
	&\le -\ep_1\int_\Omega |\Delta v|^2
	+c_2\int_\Omega |\nabla v|^{3+\eta}
	+c_2\int_\Omega |\nabla w|^{6-\frac{2\eta}{1+\eta}}
\\
&\quad\,
	+c_2\int_\Omega |\nabla w|^4
	+c_2\int_\Omega |\Delta w|^2
	+c_2\int_\Omega u^2
	-\beta\int_\Omega |\nabla v|^2
\end{align*}
%for all $t \in (0,\tmax)$ 
with some $c_2>0$, 
which along with \eqref{uL2} and \eqref{wLq} yields
\begin{align}\label{ddvL3}
		\frac{1}{2}\frac{d}{dt}\int_\Omega |\nabla v|^2
	&\le -\ep_1\int_\Omega |\Delta v|^2
	+c_2\int_\Omega |\nabla v|^{3+\eta}
	+c_2\int_\Omega |\Delta w|^2
	-\beta\int_\Omega |\nabla v|^2
	+c_3
\end{align}
%for all $t \in (0,\tmax)$ 
with some $c_3>0$. 
We now pick $q \in (n,6)$. 
Then we can verify that there exists $\eta>0$ such that 
\begin{align*}
	a_\eta := \frac{1-\frac{n}{3+\eta}+\frac{n}{q}}{2-\frac{n}{2}+\frac{n}{q}} 
	          \in \Big(\frac{1}{2}, 1\Big),
	\qquad
	(3+\eta)a_\eta<2.
\end{align*}
Hence, employing \cite[p.\ 242, Lemma~4.1]{P-1983} and 
standard elliptic regularity theory implies
\begin{align*}
  \|\nabla v(\cdot,t)\|_{L^{3+\eta}(\Omega)}^{3+\eta}
  &\le c_4\|v(\cdot,t)\|_{W^{2,2}(\Omega)}^{(3+\eta)a_\eta}
            \|v(\cdot,t)\|_{L^q(\Omega)}^{(3+\eta)(1-a_\eta)}\\
  &\le c_5\|\Delta v(\cdot,t)\|_{L^2(\Omega)}^{(3+\eta)a_\eta}
            \|v(\cdot,t)\|_{L^q(\Omega)}^{(3+\eta)(1-a_\eta)}
\end{align*}
with some $c_4, c_5>0$, which in conjunction with 
the Young inequality, the estimate \eqref{Linfestv} and 
the fact $(3+\eta)a_\eta<2$ asserts
\begin{align*}
	\|\nabla v(\cdot,t)\|_{L^{3+\eta}(\Omega)}^{3+\eta}
	\le \ep_1\|\Delta v(\cdot,t)\|_{L^2(\Omega)}^2+c_6
\end{align*}
with some $c_6>0$. 
Combining this estimate with \eqref{ddvL3} yields 
\begin{align*}
	\frac{1}{2}\frac{d}{dt}\int_\Omega |\nabla v|^2
	+\beta\int_\Omega |\nabla v|^2
	&\le 
    c_2\int_\Omega |\Delta w|^2
	+c_7
\end{align*}
%for all $t \in (0,\tmax)$ 
with some $c_7>0$. 
%which along with \eqref{wLq} leads to 
%\begin{align*}
%	\frac{1}{2}\frac{d}{dt}\int_\Omega |\nabla v|^2
%	+\beta\int_\Omega |\nabla v|^2
%	&\le 
%	c_2\int_\Omega |\Delta w|^2
%	+c_8
%\end{align*}
%%for all $t \in (0,\tmax)$ 
%with some $c_8>0$. 
Thus, thanks to \eqref{dwL2}, we arrive at the conclusion of this lemma. 
\end{proof}

\section{$L^r$-boundedness of $\nabla v$ with some $r>n$} \label{Sec5} 
In this section we will upgrade $L^2$-boundedness of $\nabla v$ to 
$L^r$-boundedness of $\nabla v$ with some $r>n$. 
To see this we take $\sigma_1, \sigma_2, \sigma_3>0$ fulfilling 
$0<\sigma_1<\sigma_2<\sigma_3<\tmax$ and put 
\begin{align*}
	M_r(T):=\sup_{t \in (\sigma_3, T)}\|\nabla v(\cdot,t)\|_{L^r(\Omega)}
\end{align*}
for $T>\sigma_3$. 
Also, we let the symbol $A$ denote the realization of 
the operator $-\Delta+\delta$
under homogeneous Neumann boundary condition in $L^2(\Omega)$. 
We then note that $A$ is sectorial and thus possesses closed fractional 
powers $A^\rho$ for arbitrary $\rho>0$, 
and the corresponding domains $D(A^\rho)$ 
have the embedding property, that is, 
$D(A^\rho) \hookrightarrow W^{2,2}(\Omega)$ 
when $2-\frac{n}{2}<2\rho-\frac{n}{2}$ (see e.g.\ \cite[Lemma~4.1]{LM-2017-2}). 
We first prove some estimate for $u$.

\begin{lem}\label{Lem:AzL2}
There exist $\rho \in (0,\frac12)$, $a \in (0,1)$ and $C>0$ 
such that for all $T \in (\sigma_3, \tmax)$, 
\begin{align}\label{AzL2}
\|A^\rho u(\cdot,t)\|_{L^2(\Omega)} 
\le C\Big(1+M_r^a(T)\Big)
\end{align}
for all $t \in (\sigma_1, T)$. 
\end{lem}

\begin{proof}
Noting that $\frac{2n}{n+2}=\frac{2}{1+\frac{2}{n}} \le \frac65<\frac32$ for $n \le 3$, 
we can find $\rho \in (0,\frac{1}{2})$ such that $\frac{2n}{n+2(1-2\rho)}<\frac{3}{2}$. 
%Indeed, when $\rho \in (0,\frac{1}{2})$, we verify that $\frac{2n}{n+2(1-2\rho)}<\frac{3}{2}$ is 
%equivalent to $6(1-2\rho)-n>0$. 
%Recalling that $n \le 3$ yields that there exists $\rho \in (0,\frac{1}{2})$ such that $6(1-2\rho)-n>0$. 
We now let $q \in (\frac{2n}{n+2(1-2\rho)}, \frac{3}{2})$ 
and $r \in (n, \min\{\frac{2q}{2-q}, \frac{2n}{(n-2)_+}\})$. 
Also, we set $\varphi:=\chi_1(v)\nabla v-\chi_2(w)\nabla w$. 
From the first equation in \eqref{ARC} we infer that for all $t \in (0,T)$, 
\begin{align}\label{Az}
	   \|A^\rho u (\cdot,t)\|_{L^2(\Omega)}
	   &\le \|A^\rho e^{-tA} u_0\|_{L^2(\Omega)}
	\notag\\
	   &\quad\, 
	          +\int_0^t \big\|A^\rho e^{-(t-s)A} 
	             \nabla \cdot \big(u(\cdot,s)\varphi(\cdot,s)\big)\big\|_{L^2(\Omega)}\,ds
	\notag\\
       &\quad\, 
              +\delta\int_0^t \|A^\rho e^{-(t-s)A} u (\cdot,s)\|_{L^2(\Omega)}\,ds.
\end{align}
Here we see from \cite[Section~2]{HW-2005} and \eqref{uL2} 
that there exists $\lam>0$ such that
\begin{align}
	           \|A^\rho e^{-tA} u_0\|_{L^2(\Omega)}
	   &\le c_1t^{-\rho} e^{-\lam t}\|u_0\|_{L^2(\Omega)}
	   \le c_2t^{-\rho} e^{-\lam t},\label{zsigma1}
	\\
	           \|A^\rho e^{-(t-s)A} u (\cdot,s)\|_{L^2(\Omega)}
       &\le c_1(t-s)^{-\rho} e^{-\lam(t-s)}\|u (\cdot, s)\|_{L^2(\Omega)}
       \le c_2(t-s)^{-\rho} e^{-\lam(t-s)}\label{zs}
\end{align}
for all $s \in (0, t)$ with some $c_1, c_2>0$. 
Also, again by \cite[Section~2]{HW-2005} and semigroup estimates, we obtain
\begin{align*}
        \big\|A^\rho e^{-(t-s)A} 
        \nabla \cdot \big(u(\cdot,s)\varphi(\cdot,s)\big)\big\|_{L^2(\Omega)}  
        &= \big\|A^\rho e^{-\frac{t-s}{2}A} \cdot e^{-\frac{t-s}{2}A}
                \nabla \cdot \big(u(\cdot,s)\varphi(\cdot,s)\big)\big\|_{L^2(\Omega)}
     \\
        &\le c_3\Big(\frac{t-s}{2}\Big)^{-\rho-\frac{n}{2}(\frac{1}{q}-\frac{1}{2})}
               e^{-\frac{\lam(t-s)}{2}}\|u(\cdot,s)\varphi(\cdot,s)\|_{L^q(\Omega)}
\end{align*}
for all $s \in (0, t)$ with some $c_3>0$. 
We now estimate $\|u(\cdot,s)\varphi(\cdot,s)\|_{L^q(\Omega)}$. 
We observe from the H\"older inequality and the interpolation inequality that
\begin{align}\label{upLq}
	          \|u(\cdot,s)\varphi(\cdot,s)\|_{L^q(\Omega)} 
	   &\le \|u (\cdot,s)\|_{L^2(\Omega)}\|\varphi(\cdot,s)\|_{L^{\frac{2q}{2-q}}(\Omega)}
	\notag\\
	   &\le \|u (\cdot,s)\|_{L^2(\Omega)}
	          \|\varphi(\cdot,s)\|_{L^r(\Omega)}^a
	          \|\varphi(\cdot,s)\|_{L^2(\Omega)}^{1-a},
\end{align}
where $a:=\frac{\frac{1}{2}-\frac{2-q}{2q}}{\frac{1}{2}-\frac{1}{r}} \in (0,1)$. 
Here, noting from \eqref{dchi1} that $\chi_1(s) \le \chi_1(0)$, $\chi_2(s) \le \chi_2(0)$, 
we see from \eqref{wLq} that 
\begin{align}\label{pLr}
	\|\varphi(\cdot,s)\|_{L^r(\Omega)}
	&\le c_4\big(\|\nabla v(\cdot,s)\|_{L^r(\Omega)}+\|\nabla w(\cdot,s)\|_{L^r(\Omega)}\big)
	\notag\\
	&\le c_5\big(\|\nabla v(\cdot,s)\|_{L^r(\Omega)}+1\big)
\end{align}
for all $s \in (0, t)$ with some $c_4, c_5>0$. 
Similarly, we derive from \eqref{dwL21} and \eqref{dvL2est} that
\begin{align}\label{pL2}
	\|\varphi(\cdot,s)\|_{L^2(\Omega)} \le c_6
\end{align}
for all $s \in (0, t)$ with some $c_6>0$. 
Combining \eqref{pLr} and \eqref{pL2} with \eqref{upLq} 
and recalling the definition of $M_r(T)$ imply
\begin{align}\label{zphi1}
	        \|u(\cdot,s)\varphi(\cdot,s)\|_{L^q(\Omega)} 
	&\le c_7\Big(1+M_r^a(T)\Big)
\end{align}
for all $s \in (0, t)$ with some $c_7>0$. 
On the other hand, 
noting from Lemma~\ref{LSE} that $\sup_{s \in (0,\sigma_3)}\|\nabla v(\cdot,s)\|_{L^r(\Omega)}<\infty$, 
we have from \eqref{pLr} that 
$\sup_{s \in (0,\sigma_3)}\|\varphi(\cdot,s)\|_{L^r(\Omega)}<\infty$, 
which together with \eqref{upLq} that
\begin{align}\label{zphi11}
	\|u(\cdot,s)\varphi(\cdot,s)\|_{L^q(\Omega)}  \le c_8
\end{align}
for all $s \in (0,\sigma_3)$ with some $c_8>0$. 
Collecting \eqref{zsigma1}, \eqref{zs}, \eqref{zphi1} and \eqref{zphi11} 
into \eqref{Az} yields
\begin{align*}
	           \|A^\rho u (\cdot,t)\|_{L^2(\Omega)}
	   &\le  c_2t^{-\rho} e^{-\lam t}
	           +c_9\Big(1+M_r^a(T)\Big)
	            \int_0^t\Big(\frac{t-s}{2}\Big)^{-\rho-\frac{n}{2}(\frac{1}{q}-\frac{1}{2})}
	            e^{-\frac{\lam(t-s)}{2}}\,ds
	\\
	   &\quad\,
	   		   +c_2\delta\int_0^t(t-s)^{-\rho} e^{-\lam(t-s)}\,ds
\end{align*}
for all $t \in (0, T)$ with some $c_9>0$. 
Noticing from the choice of $\rho$ that $\rho+\frac{n}{2}(\frac{1}{q}-\frac{1}{2})<1$, 
we see that 
\begin{align*}
				\|A^\rho u (\cdot,t)\|_{L^2(\Omega)}
		&\le  c_{10}\sigma_1^{-\rho}
	    	    +c_{11}\Big(1+M_r^a(T)\Big)
	      \le c_{12}\Big(1+M_r^a(T)\Big)
\end{align*}
for all $t \in (0, T)$ with some $c_{10}, c_{11}, c_{12}>0$, which means that \eqref{AzL2} holds. 
\end{proof}

We now derive $L^2$-boundedness of $\Delta w$. 

\begin{lem}\label{Lem:Lapw}
%Assume that $\chi_1, \chi_2, \xi$ fulfill \eqref{chi1class}--\/\eqref{dxi} 
%with  $K_1, K_2$ which satisfy \eqref{condipara}. 
There exists $C>0$ such that for all $T \in (\sigma_2, \tmax)$, 
\begin{align}\label{LapwL2}
\|\Delta w(\cdot,t)\|_{L^2(\Omega)} \le C\Big(1+M_r^a(T)\Big)
\end{align}
for all $t \in (\sigma_2, T)$, 
where $a \in (0,1)$ is obtained in {\rm Lemma~\ref{Lem:AzL2}}. 
\end{lem}

\begin{proof}
Let $\rho \in (0,\frac{1}{2})$ be as in Lemma~\ref{Lem:AzL2} 
and let $\eta \in (1, 1+\rho)$. 
Then we note that $2-\frac{n}{2}<2\rho-\frac{n}{2}$. 
By a similar argument as in the proof of \cite[Lemma~4.3]{LM-2017-2} 
and Lemma~\ref{Lem:AzL2}, we obtain
\begin{align*}
	       \|w(\cdot,t)\|_{W^{2,2}(\Omega)}
	&\le c_1\|A^\eta 
	w(\cdot,t) \|_{L^2(\Omega)}
	\\
	&\le c_1\|A^\eta 
	e^{-(t-\sigma_1)A}w(\cdot,\sigma_1) \|_{L^2(\Omega)}
	       +c_1\int_{\sigma_1}^t \|A^\eta 
	       e^{-(t-s)A} u(\cdot,s) \|_{L^2(\Omega)}\,ds
	\\
	&\le c_2(t-\sigma_1)^{-\eta
	}\|w(\cdot,\sigma_1) \|_{L^2(\Omega)}
	       +c_2\int_{\sigma_1}^t \|A^{\eta
	       	-\rho} e^{-(t-s)A} \cdot 
	                                            A^\rho u(\cdot,s) \|_{L^2(\Omega)}\,ds
	\\
	&\le c_3(t-\sigma_1)^{-\eta
	}+c_2\int_{\sigma_1}^t (t-s)^{-(\eta
	-\rho)}e^{-\lam t}
	\|A^\rho u(\cdot,s) \|_{L^2(\Omega)}\,ds
	\\
	&\le c_4\Big(1+(\sigma_2-\sigma_1)^{-\eta
	}+M_r^a(T)\Big)
\end{align*}
for all $t \in (\sigma_2, T)$ with some $c_1, c_2, c_3, c_4>0$, 
which leads to \eqref{LapwL2}. 
\end{proof}

Finally, we prove $L^r$-boundedness of $\nabla v$.

\begin{lem}\label{Lem:dvLr}
%Assume that $\chi_1, \chi_2, \xi$ fulfill \eqref{chi1class}--\/\eqref{dxi} 
%with  $K_1, K_2$ which satisfy \eqref{condipara}. 
There exists $r>n$ such that 
\begin{align}\label{dvLrest}
\|\nabla v(\cdot,t)\|_{L^r(\Omega)} \le C
\end{align}
for all $t \in (0, \tmax)$ with some $C>0$. 
\end{lem}

\begin{proof}
Let $q \in (n,\frac{2n}{(n-2)_+})$, $r\in (2, \frac{2n}{(n-2)_+})$. 
%We put $N_r(T):=\sup_{t \in (0,T)} \|\nabla v(\cdot,t)\|_{L^r(\Omega)}$ for $T \in (0,\tmax)$. 
We will find a constant $c_1>0$ such that 
$\sup_{t \in (\sigma_3,T)} \|\nabla v(\cdot,t)\|_{L^r(\Omega)} \le c_1$ 
for all $T \in (\sigma_3,\tmax)$. 
%To this end we let the symbol $(e^{t\Delta})_{t\ge0}$ denote 
%the Neumann heat semigroup in $\Omega$. 
We now rewrite the second equation in \eqref{ARC} by the Duhamel formula as
\begin{align*}
	v(\cdot,t)
	&= e^{(t-\sigma_3)(\Delta-\beta)}v(\cdot,\sigma_3)
	-\int_{\sigma_3}^t e^{(t-s)(\Delta-\beta)}
	\nabla \cdot \Big(v(\cdot,s)\xi(w(\cdot,s))\nabla w(\cdot,s)\Big)\,ds\\
	&\quad\,
	+\alpha \int_{\sigma_3}^t e^{(t-s)(\Delta-\beta)}u(\cdot,s)\,ds
\end{align*}
for all $t \in (\sigma_3, T)$, which implies 
\begin{align}\label{dvLr}
	\|\nabla v(\cdot,t)\|_{L^r(\Omega)}
	&= \|\nabla e^{(t-\sigma_3)(\Delta-\beta)}v(\cdot,\sigma_3)\|_{L^r(\Omega)}
	\notag\\
	&\quad\,
	+\int_{\sigma_3}^t \Big\|\nabla e^{(t-s)(\Delta-\beta)}
	\nabla \cdot \Big(v(\cdot,s)\xi(w(\cdot,s))\nabla w(\cdot,s)\Big)\Big\|_{L^r(\Omega)}\,ds
	\notag\\
	&\quad\,
	+\alpha \int_{\sigma_3}^t \|\nabla e^{(t-s)(\Delta-\beta)}u(\cdot,s)\|_{L^r(\Omega)}\,ds
	\notag\\
	&=: \I_1(\cdot,t)+\I_2(\cdot,t)+\I_3(\cdot,t)
\end{align}
for all $t \in (\sigma_3, T)$. 
We first estimate $\I_1$. 
Using semigroup estimates, we can estimate as
\begin{align}\label{I1estimate}
	\I_1(\cdot,t) \le c_2
\end{align}
for all $t \in (\sigma_3, T)$ with some $c_2>0$. 
%We first estimate $\I_1$. 
%If $t \le 1$, then $\sigma_4=0$. 
%Thus we see from semigroup estimates that
%\begin{align*}
%	      \I_1(\cdot,t)
%	&= \|\nabla e^{t(\Delta-\beta)}v_0\|_{L^r(\Omega)}
%	\\
%	&\le c_2(1+t^{-\frac{1}{2}-\frac{n}{2}(1-\frac{1}{r})})e^{-\lam_1t}\|v_0\|_{L^1(\Omega)}
%	\\
%	&\le c_3\|v_0\|_{L^1(\Omega)}
%\end{align*}
%for all $t \in (0, T)$ with some $c_2, c_3>0$, 
%where $\lam_1>0$ is the first nonzero eigenvalue of $-\Delta$ in $\Omega$ 
%under Neumann boundary conditions. 
%Also, if $t>1$, then $\sigma_4=t-1$, which implies from 
%$L^\infty$-boundedness of $v$ (see Lemma~\ref{Lem:vLinf}) that 
%\begin{align*}
%	\I_1(\cdot,t)
%	&= \|\nabla e^{\Delta-\beta}v(\cdot, t-1)\|_{L^r(\Omega)}
%	\\
%	&\le c_4\|v(\cdot, t-1)\|_{L^r(\Omega)}
%	\\
%	&\le c_5
%\end{align*}
%for all $t \in (0, T)$ with some $c_4, c_5>0$. 
%Thus we have
%\begin{align}\label{I_1estimate}
%	\I_1(\cdot,t) \le \max\{c_3\|v_0\|_{L^1(\Omega)}, c_5\}
%\end{align}
%for all $t \in (0, T)$. 
We next estimate $\I_2$. 
Let $\max\{\frac{n}{2}, \frac{2q}{q+2}, \frac{nr}{n+r}\}<k<\min\{\frac{n}{n-2}, \frac{qr}{q+r}, 2\}$. 
Then we derive from semigroup estimates that 
\begin{align}\label{I_2est}
	     \I_2(\cdot,t)
	\le c_3\int_{\sigma_3}^t (1+(t-s)^{-\frac{1}{2}-\frac{n}{2}(\frac{1}{k}-\frac{1}{r})})
	     \Big\|\nabla \cdot \Big(v(\cdot,s)\xi(w(\cdot,s))\nabla w(\cdot,s)\Big)\Big\|_{L^k(\Omega)}\,ds
\end{align}
for all $t \in (\sigma_3, T)$ with some $c_3>0$. 
Here, noting from the conditions \eqref{dchi1}, \eqref{xi} that 
$\xi(s) \le \xi(0)$, we observe from 
the condition \eqref{dxi} and the estimate \eqref{Linfestv} that
\begin{align}\label{I_2int}
	&\Big\|\nabla \cdot \Big(v(\cdot,s)\xi(w(\cdot,s))\nabla w(\cdot,s)\Big)\Big\|_{L^k(\Omega)}
	\notag\\
	&\quad\,\le
	  \|\xi(w(\cdot,s))\nabla v(\cdot,s) \cdot \nabla w(\cdot,s)\|_{L^k(\Omega)}
	  +\|v(\cdot,s)\xi'(w(\cdot,s))|\nabla w(\cdot,s)|^2\|_{L^k(\Omega)}
	\notag\\
	&\qquad\ \,
	  +\|v(\cdot,s)\xi(w(\cdot,s))\Delta w(\cdot,s)\|_{L^k(\Omega)}
	\notag\\
	&\quad\,\le c_4\|\nabla v(\cdot,s) \cdot \nabla w(\cdot,s)\|_{L^k(\Omega)}
	                  +c_4\big\||\nabla w(\cdot,s)|^2\big\|_{L^k(\Omega)}
	                  +c_4\|\Delta w(\cdot,s)\|_{L^k(\Omega)}
\end{align}
for all $s \in (\sigma_3, t)$ with some $c_4>0$. 
We now estimate three terms appearing on the right-hand side of 
this inequality. 
Noticing that $\frac{n}{2}<k<\frac{n}{n-2}$, we see from \eqref{wLq} that 
\begin{align}\label{dwLk2}
	\big\||\nabla w(\cdot,s)|^2\big\|_{L^k(\Omega)} \le c_5
\end{align}
for all $s \in (\sigma_3, t)$ with some $c_5>0$. 
Also, since $k<2$, we infer from \eqref{LapwL2} that
\begin{align}\label{LapwLk2}
	\|\Delta w(\cdot,s)\|_{L^k(\Omega)} 
	\le c_6\sup_{t \in (\sigma_3, T)}\|\nabla v(\cdot,t)\|_{L^r(\Omega)}^a
\end{align}
for all $s \in (\sigma_3, t)$ with some $c_6>0$
where $a \in (0,1)$ is a constant obtained in Lemma~\ref{Lem:AzL2}. 
We next estimate $\|\nabla v(\cdot,s) \cdot \nabla w(\cdot,s)\|_{L^k(\Omega)}$. 
The H\"older inequality and the estimate \eqref{wLq} lead to
\begin{align}\label{dvdw}
	     \|\nabla v(\cdot,s) \cdot \nabla w(\cdot,s)\|_{L^k(\Omega)}
	&\le \|\nabla v(\cdot,s)\|_{L^{\frac{kq}{q-k}}(\Omega)}
	     \|\nabla w(\cdot,s)\|_{L^q(\Omega)}
	\notag\\
	&\le c_7\|\nabla v(\cdot,s)\|_{L^{\frac{kq}{q-k}}(\Omega)}
\end{align}
for all $s \in (\sigma_3, t)$ with some $c_7>0$, where $q \in (n, \frac{2n}{n-2})$. 
Moreover, since $\frac{2q}{q+2}<k<\frac{qr}{q+r}$, 
the interpolation inequality and the estimate \eqref{dvL2est} imply
\begin{align}\label{dvkq}
	     \|\nabla v(\cdot,s)\|_{L^{\frac{kq}{q-k}}(\Omega)}
	\le \|\nabla v(\cdot,s)\|_{L^2(\Omega)}^{\theta_1}\|\nabla v(\cdot,s)\|_{L^r(\Omega)}^{1-\theta_1}
	\le c_8\|\nabla v(\cdot,s)\|_{L^r(\Omega)}^{1-\theta_1}
\end{align}
for all $s \in (\sigma_3,t)$ with some $c_8>0$, 
where $\theta_1:=\frac{2[r(q-k)-kq]}{kq(r-2)} \in (0,1)$. 
A combination of \eqref{dvdw} and \eqref{dvkq} yields
\begin{align}\label{dvdw2}
	\|\nabla v(\cdot,s) \cdot \nabla w(\cdot,s)\|_{L^k(\Omega)}
	\le c_9\|\nabla v(\cdot,s)\|_{L^r(\Omega)}^{1-\theta_1}
\end{align}
for all $s \in (\sigma_3,t)$ with some $c_9>0$. 
Combining \eqref{dwLk2}, \eqref{LapwLk2} and \eqref{dvdw2} with \eqref{I_2int}, 
we have
\begin{align}\label{I_2int2}
	\Big\|\nabla \cdot \Big(v(\cdot,s)\xi(w(\cdot,s))\nabla w(\cdot,s)\Big)\Big\|_{L^k(\Omega)}
	\le c_{10}\sup_{t \in (\sigma_3, T)}\|\nabla v(\cdot,t)\|_{L^r(\Omega)}^{\theta_2}
	     +c_{10}
\end{align}
for all $s \in (\sigma_3,t)$ with some $c_{10}>0$, where $\theta_2:=\max\{a, 1-\theta_1\} \in (0,1)$. 
Furthermore, noting from $k>\frac{nr}{n+r}$ that 
$\frac{1}{2}+\frac{n}{2}(\frac{1}{k}-\frac{1}{r})<1$, 
we derive from \eqref{I_2est} and \eqref{I_2int2} that
\begin{align}\label{I_2estimate}
	\I_2(\cdot,t) 
	\le c_{11}\sup_{t \in (\sigma_3, T)}\|\nabla v(\cdot,t)\|_{L^r(\Omega)}^{\theta_2}
	     +c_{11}
\end{align}
for all $t \in (\sigma_3, T)$ with some $c_{11}>0$. 
We next estimate $\I_3$. 
Again by semigroup estimates and Lemma~\ref{Lem:uL2}, we infer
\begin{align*}
	\I_3(\cdot,t)
	&\le c_{12}\int_{\sigma_3}^t
	               (1+(t-s)^{-\frac{1}{2}-\frac{n}{2}(\frac{1}{2}-\frac{1}{r})})
	               e^{-\lam_1(t-s)}\|u(\cdot,s)\|_{L^2(\Omega)}\,ds
	\\
	&\le c_{13}\int_{\sigma_3}^t
	(1+(t-s)^{-\frac{1}{2}-\frac{n}{2}(\frac{1}{2}-\frac{1}{r})})
	e^{-\lam_1(t-s)}\,ds
\end{align*}
for all $t \in (\sigma_3, T)$ with some $c_{12}, c_{13}>0$. 
Noticing from $r\in (2, \frac{2n}{(n-2)_+})$ that 
$\frac{1}{2}+\frac{n}{2}(\frac{1}{2}-\frac{1}{r})<1$, we have
\begin{align}\label{I_3estimate}
	\I_3(\cdot,t)\le c_{14}
\end{align}
for all $t \in (\sigma_3, T)$ with some $c_{14}>0$. 
Collecting \eqref{I1estimate}, \eqref{I_2estimate} and \eqref{I_3estimate} 
into \eqref{dvLr} yields
\begin{align*}
	\sup_{t \in (\sigma_3,T)} \|\nabla v(\cdot,t)\|_{L^r(\Omega)} 
	\le c_{11}\sup_{t \in (\sigma_3,T)} \|\nabla v(\cdot,t)\|_{L^r(\Omega)}^{1-\theta_2}+c_{15}
\end{align*}
for all $T \in (\sigma_3, \tmax)$ with some $c_{15}>0$, which implies 
$\sup_{t \in (\sigma_3,T)} \|\nabla v(\cdot,t)\|_{L^r(\Omega)} \le c_{16}$ for all $T \in (\sigma_3,\tmax)$ with some $c_{16}>0$. 
Also, in view of Lemma~\ref{LSE} it follows that $\sup_{t \in (0,\sigma_3)}\|\nabla v(\cdot,t)\|_{L^r(\Omega)} \le c_{17}$ with some $c_{17}>0$. 
Thus we arrive at the conclusion. 
\end{proof}

\section{$L^\infty$-boundedness of $u$} \label{Sec6}

We now obtain $L^\infty$-boundedness of $u$.

\begin{lem}\label{Lem:uLinf}
%Assume that $\chi_1, \chi_2, \xi$ fulfill \eqref{chi1class}--\/\eqref{dxi} 
%with  $K_1, K_2$ which satisfy \eqref{condipara}. 
There exists $C>0$ such that
\begin{align*}
\|u(\cdot,t)\|_{L^\infty(\Omega)} \le C
\end{align*}
for all $t \in (0, \tmax)$. 
\end{lem}

\begin{proof}
Proceeding similarly as in the proof of Lemma~\ref{Lem:vLinf}
and using boundedness of $\chi_1, \chi_2, \xi$ as well as 
recalling the estimates \eqref{wLq} and \eqref{dvLrest}, 
we can arrive at the conclusion of this lemma. 
\end{proof}

We are now in the position to prove Theorem~\ref{mainthm}.

\begin{prth1.1}
 Let $\chi_1, \chi_2, \xi$ satisfy \eqref{chi1class}--\/\eqref{dxi} 
 with $K_1, K_2$ fulfilling \eqref{condipara}.  Then, a combination of Lemma~\ref{Lem:uLinf} and 
 the extensibility criterion \eqref{BU} leads to the end of the proof.
 \qed
\end{prth1.1}
%
%==========================================================
%%%%%%%                                             %%%%%%%
  %%%                                                 %%%
 %%%                                                   %%%
%%%                    Acknowledgments                  %%%
 %%%                                                   %%%
  %%%                                                 %%%
%%%%%%%                                             %%%%%%%
%==========================================================
%\smallskip
%\red{
%\section*{Acknowledgments}
%The authors would like to thank the referees 
%for their helpful suggestions on improving this paper.
%Johannes
%}
%%==============================================================%%
%%==============                                  ==============%%
%%======                                                  ======%%
%%====                                                      ====%%
%%==                         Reference                        ==%%
%%======                                                  ======%%
%%==============                                  ==============%%
%%==============================================================%%

%\bibliography{ref}
%\bibliographystyle{plain}
\end{document}